\documentclass{amsart}
\usepackage{amsmath,amsthm,enumerate,nccmath}

\theoremstyle{theorem}
\newtheorem{theorem}{Theorem}
\newtheorem{corollary}[theorem]{Corollary}
\newtheorem{lemma}[theorem]{Lemma}
\newtheorem{proposition}[theorem]{Proposition}

\theoremstyle{definition}

\newcommand{\tr}{\mathrm{tr}\,}
\newcommand{\vy}[1]{{\left\vert\kern-0.25ex\left\vert\kern-0.25ex\left\vert #1 
    \right\vert\kern-0.25ex\right\vert\kern-0.25ex\right\vert}}
    
\title[Mixing properties of matrix equilibrium states]{A necessary and sufficient condition for a matrix equilibrium state to be mixing}

\begin{document}
\author{Ian D. Morris}

\address{Department of Mathematics, University of Surrey, Guildford GU2 7XH, U.K.}
\email{i.morris@surrey.ac.uk}
\begin{abstract}
Since the 1970s there has been a rich theory of equilibrium states over shift spaces associated to H\"older-continuous real-valued potentials. The construction of equilibrium states associated to matrix-valued potentials is much more recent, with a complete description of such equilibrium states being achieved in 2011 by D.-J. Feng and A. K\"aenm\"aki. In a recent article the author investigated the ergodic-theoretic properties of these matrix equilibrium states, attempting in particular to give necessary and sufficient conditions for mixing, positive entropy, and the property of being a Bernoulli measure with respect to the natural partition, in terms of the algebraic properties of the semigroup generated by the matrices. Necessary and sufficient conditions were successfully established for the latter two properties but only a sufficient condition for mixing was given. The purpose of this note is to complete that investigation by giving a necessary and sufficient condition for a matrix equilibrium state to be mixing.
 \end{abstract}
\maketitle
\section{Introduction}

For each $N \geq 2$ let us define $\Sigma_N:=\{1,\ldots,N\}^{\mathbb{N}}$ and equip this space with the infinite product topology. Let $\sigma \colon \Sigma_N \to \Sigma_N$ denote the shift transformation $\sigma[(x_i)_{i=1}^\infty]:=(x_i)_{i=1}^\infty$, which is a continuous transformation of $\Sigma_N$, and let $\mathcal{M}_\sigma$ denote the set of all $\sigma$-invariant Borel probability measures on $\Sigma_N$. If $f \colon \Sigma_N \to \mathbb{R}$ is a continuous function we say that $\mu \in \mathcal{M}_\sigma$ is an equilibrium state for $f$ if it satisfies $h(\mu)+\int f \,d\mu \geq h(\nu)+\int f\,d\nu$ for all $\nu \in \mathcal{M}_\sigma$, where $h$ denotes metric entropy. If $f$ is H\"older continuous with respect to a natural metric on $\Sigma_N$ then it is known that $f$ has a unique equilibrium state, and this equilibrium state enjoys numerous regularity properties such as full support, mixing, positive entropy and a Gibbs inequality; see \cite{PaPo90} for details. 

Equilibrium states of continuous functions $f \colon \Sigma_N \to \mathbb{R}$ have been extensively investigated since the 1970s due to their broad and deep applications to smooth ergodic theory. More recently interest has developed in equilibrium states defined by matrix-valued potentials, primarily due to their applications in the dimension theory of self-affine fractals where they play a pivotal role, particularly in the planar case (see e.g. \cite{BaKaKo16,FaKe16,HuLa95,MoSh16}). Given an $N$-tuple of $d \times d$ matrices $\mathsf{A}=(A_1,\ldots,A_N)$ and a real number $s>0$ we define the \emph{pressure of $\mathsf{A}$ with parameter $s$} to be the limit
\begin{equation}\label{eq:pressure}P(\mathsf{A},s):=\lim_{n\to\infty}\frac{1}{n}\log \left(\sum_{i_1,\ldots,i_n=1}^N \|A_{i_n}\cdots A_{i_1}\|^s\right) \in [-\infty,+\infty)\end{equation}
which exists by subadditivity. For every $\mu \in \mathcal{M}_\sigma$ we define the \emph{top Lyapunov exponent of $\mathsf{A}$ with respect to $\mu$} to be the quantity
\begin{align}\label{eq:lyap}\Lambda(\mathsf{A},\mu)&:=\lim_{n\to\infty}\frac{1}{n}\int \log\|A_{x_n}\cdots A_{x_1}\|d\mu\left[(x_i)_{i=1}^\infty\right]\\\nonumber
&=\inf_{n \geq 1}\frac{1}{n}\int \log\|A_{x_n}\cdots A_{x_1}\|d\mu\left[(x_i)_{i=1}^\infty\right]\end{align}
again using subadditivity. By the subadditive variational principle \cite{CaFeHu08} we have
\[P(\mathsf{A},s)=\sup_{\mu \in \mathcal{M}_\sigma}\left[ h(\mu)+s\Lambda(\mathsf{A},\mu)\right]\]
for every $s>0$ and every $N$-tuple of $d \times d$ matrices $\mathsf{A}$. If a measure $\mu$ attains this supremum then it is called an \emph{equilibrium state} of $(\mathsf{A},s)$. Since the entropy and Lyapunov exponent are both affine functions of the measure $\mu$ when all other parameters are fixed, the set of all equilibrium states of $(\mathsf{A},s)$ is convex. Furthermore it is clear that a linear combination of invariant probability measures $\sum_{i=1}^m \alpha_i\mu_i \in \mathcal{M}_\sigma$ can only be an equilibrium state for a given $(\mathsf{A},s)$ if all of its components $\mu_i$ are equilibrium states for $(\mathsf{A},s)$.

Throughout this article we let $M_d(\mathbb{R})$ denote the vector space of all $d \times d$ real matrices and $\|\cdot\|$ the Euclidean norm. We say that a tuple $\mathsf{A}=(A_1,\ldots,A_N)\in M_d(\mathbb{R})^N$ is \emph{irreducible} if there does not exist a proper nonzero linear subspace $U$ of $\mathbb{R}^d$ such that $A_iU \subseteq U$ for every $i=1,\ldots,N$. If such a subspace exists then we call $U$ an $\mathsf{A}$-invariant subspace and say that $\mathsf{A}$ is \emph{reducible}. In a similar manner we will say that a subset of $M_d(\mathbb{R})$ is reducible if all its elements preserve a common nonzero proper linear subspace of $\mathbb{R}^d$, and irreducible otherwise. The basic properties of equilibrium states of $\mathsf{A}=(A_1,\ldots,A_N)\in M_d(\mathbb{R})^N$ are described by the following result of D.-J. Feng and A. K\"aenmaki \cite{FeKa10}:
\begin{theorem}\label{th:fk}
Let $\mathsf{A}=(A_1,\ldots,A_N) \in M_d(\mathbb{R})$ and $s>0$. If $\mathsf{A}$ is irreducible then $P(\mathsf{A},s)>-\infty$, there exists a unique equilibrium state $\mu \in \mathcal{M}_\sigma$ of $(\mathsf{A},s)$, and $\mu$ is ergodic with respect to $\sigma$. Furthermore there exists a constant $C>0$ such that for every $x_1,\ldots,x_n \in \{1,\ldots,N\}$
\[C^{-1}\mu([x_1\cdots x_n])\leq \frac{\|A_{x_n}\cdots A_{x_1}\|^s}{e^{nP(\mathsf{A},s)}} \leq C\mu([x_1\cdots x_n])\]
where
\[[x_1\cdots x_n]:=\left\{(z_i)_{i=1}^\infty \in \Sigma_N \colon x_i=z_i \text{ for all }i=1,\ldots,n\right\}.\]
If $\mathsf{A}$ is reducible and $P(\mathsf{A},s)>-\infty$ then there exist an invertible matrix $X$, natural numbers $d_1,\ldots,d_\ell$ such that $\sum_{i=1}^\ell d_i=d$, and matrices $A_i^{(j,k)}$ such that 
\[
X^{-1}A_iX = \left(\begin{array}{ccccc}
A_i^{(1,1)} & A_i^{(1,2)} & A_i^{(1,3)} & \cdots &A_i^{(1,\ell)}\\
0 & A_i^{(2,2)} & A_i^{(2,3)} & \cdots &A_i^{(2,\ell)}\\
0 & 0 & A_i^{(3,3)} & \cdots &A_i^{(3,\ell)}\\
\vdots & \vdots & \vdots  & \ddots &\vdots \\
0 & 0 & 0  & \cdots &A_i^{(\ell,\ell)}
\end{array}\right)\]
for each $i=1,\ldots,N$, where each  matrix $A_i^{(j,k)}$ has dimensions $d_j \times d_k$ and for each $j=1,\ldots,\ell$ the tuple $\mathsf{A}^{(j)}=(A_1^{(j,j)},\ldots,A_N^{(j,j)})\in M_{d_j}(\mathbb{R})^N$ is irreducible. The ergodic equilibrium states of $(\mathsf{A},s)$ are precisely the equilibrium states of those tuples $(\mathsf{A}^{(i)},s)$ such that $P(\mathsf{A}^{(i)},s)=P(\mathsf{A},s)$, and at least one tuple with this property exists. In particular the number of ergodic equilibrium states of $(\mathsf{A},s)$ is at least $1$ and is at most $\ell$.
\end{theorem}

In the previous article \cite{Mo17} the author investigated the question of when the equilibrium state of a finite set of $d \times d$ matrices and a parameter $s>0$ has various ergodic-theoretic properties, in particular total ergodicity, mixing, full support, positive entropy, and the property of being a Bernoulli measure with respect to the natural partition. In most of these cases necessary and sufficient conditions for the property concerned were given in terms of the algebraic properties of the semigroup generated by the matrices. It was shown in that article that total ergodicity of matrix equilibrium states implies mixing, and that if the equilibrium state is not totally ergodic then there exists an integer $n>1$ such that the set
\[\{A_{i_n}\cdots A_{i_1}\colon 1 \leq i_1,\ldots,i_n \leq N\}\]
preserves a nonzero proper subspace of $\mathbb{R}^d$. However, a necessary and sufficient algebraic condition for mixing was not found. It was conjectured in \cite{Mo17} that if an equilibrium state associated to an irreducible tuple of matrices is ergodic but not totally ergodic with respect to $\sigma$, then it must fail to be ergodic with respect to $\sigma^d$. In this note we present a necessary and sufficient algebraic condition for an ergodic matrix equilibrium state to be mixing, and prove the aforementioned conjecture, both subject to the additional condition that at least one of the matrices is invertible. In the case where all of the matrices are not invertible we are not able to completely characterise mixing, but we are able to show in this case that the conjecture fails.

If $\mathsf{A} \in M_d(\mathbb{R})^N$ is irreducible and $s>0$, then by Theorem \ref{th:fk} there is a unique equilibrium state $\mu$ for $(\mathsf{A},s)$ and this equilibrium state is ergodic with respect to $\sigma$. For each $n\geq 1$ we let $\mathsf{A}_n$ denote the $N^n$-tuple of all products $A_{i_n}\cdots A_{i_1}$ listed in lexicographical order. A priori, by Theorem \ref{th:fk} we may find a finite collection of equilibrium states $\nu_1,\ldots,\nu_\ell$ for $(\mathsf{A}_n,s)$ which are measures defined on $\Sigma_{N^n}$ and are invariant with respect to the shift on $\Sigma_{N^n}$. However, in this article we will find it much more illuminating to identify these measures $\nu_1,\ldots,\nu_\ell$ with $\sigma^n$-invariant measures on the original space $\Sigma_N$ in the obvious fashion, and this identification will be presumed throughout much of the article. It is clear from the definition \eqref{eq:pressure} that $P(\mathsf{A}_n,s)=nP(\mathsf{A},s)$, and using \eqref{eq:lyap} it is not difficult to see that any equilibrium state of $(\mathsf{A},s)$ must, interpreted via the aforementioned identification, also be an equilibrium state of $(\mathsf{A}_n,s)$ for every $n\geq 1$. Conversely if an equilibrium state $\nu$ of $(\mathsf{A}_n,s)$ -- understood as a $\sigma^n$-invariant measure on $\Sigma_N$ -- is also invariant with respect to $\sigma$ then it must be an equilibrium state of $(\mathsf{A},s)$. If the unique equilibrium state $\mu$ of $(\mathsf{A},s)$ is not totally ergodic then for some $n>1$ it by definition fails to be ergodic with respect to $\sigma^n$, and its distinct ergodic components with respect to $\sigma^n$ then correspond to distinct $\sigma^n$-invariant equilibrium states of $(\mathsf{A}_n,s)$. By Theorem \ref{th:fk} it must therefore be the case that $\mathsf{A}_n$ is not irreducible. The failure of mixing for an equilibrium state of $(\mathsf{A},s)$ is thus intimately connected with the irreducibility properties of $\mathsf{A}_n$ for some critical integer $n>1$ and with the class of $\sigma^n$-invariant measures on $\Sigma_N$ which correspond to equilibrium states of $(\mathsf{A}_n,s)$. We see moreover that any such non-mixing equilibrium state $\mu$ must somehow be expressible as a $\sigma$-invariant average of the various $\sigma^n$-invariant equilibrium states of $(\mathsf{A}_n,s)$. These observations find their full realisation in the following theorem:
\begin{theorem}\label{th:one}
Let $\mathsf{A}=(A_1,\ldots,A_N)\in M_d(\mathbb{R})^N$ be irreducible and suppose that at least one of the matrices $A_i$ is invertible. Let $s>0$ and let $\mu$ be the unique equilibrium state of $(\mathsf{A},s)$. Then $\mu$ is \emph{not} mixing with respect to $\sigma$ if and only if the following properties hold. There exist integers $\ell>1$, $k \geq 1$ such that $k\ell=d$ and matrices $A_i^{(j)} \in M_k(\mathbb{R})$ for $i=1,\ldots,N$ and $j=1,\ldots,\ell$ and an invertible matrix $X \in M_d(\mathbb{R})$ such that
\begin{equation}\label{eq:form}X^{-1}A_iX = \left(\begin{array}{cccccc}
0&0&0& 0&\cdots & A_i^{(\ell)}\\
A_i^{(1)}&0&0&0& \cdots &0\\
0& A_i^{(2)}&0&0& \cdots &0\\
 0& 0& A_i^{(3)}&0& \cdots &0\\
\vdots&\vdots&\vdots&\ddots & \ddots&\vdots \\
 0&0& 0&\cdots &A_i^{(\ell-1)}&0
\end{array}\right)\end{equation}
for $i=1,\ldots,N$. For every $(i_1,\ldots,i_\ell)\in \{1,\ldots,N\}^\ell$ the matrix $X^{-1}A_{i_\ell}\cdots A_{i_1}X$ has the form
\begin{equation}\label{eq:ftoomch}\begin{medsize}
 \left(\begin{array}{cccc}
A_{i_\ell}^{(\ell)}A_{i_{\ell-1}}^{(\ell-1)} \cdots A_{i_2}^{(2)}A_{i_1}^{(1)}&0&\cdots &0\\
0 &  A_{i_\ell}^{(1)}A_{i_{\ell-1}}^{(\ell)} \cdots A_{i_2}^{(3)}A_{i_1}^{(2)}& \cdots &0\\
\vdots&\vdots &\ddots&\vdots \\
0 & 0& \cdots &A_{i_\ell}^{(\ell-1)}A_{i_{\ell-1}}^{(\ell-2)} \cdots A_{i_2}^{(1)} A_{i_1}^{(\ell)}
\end{array}\right).\end{medsize}\end{equation}For each $j=1,\ldots,\ell$ the set
\[\left\{A_{i_\ell}^{(j)} A_{i_{\ell-1}}^{(j-1)}\cdots A_{i_{\ell-j+1}}^{(1)}A_{i_{\ell-j}}^{(\ell)}\cdots A_{i_1}^{(j+1)} \colon 1 \leq i_1,\ldots,i_\ell \leq N\right\}\subset M_k(\mathbb{R})\]
is irreducible and $(\mathsf{A}_\ell,s)$ admits exactly $\ell$ distinct ergodic $\sigma^\ell$-invariant equilibrium states $\mu_1,\ldots,\mu_\ell$, one corresponding to each of the $\ell$ diagonal blocks in \eqref{eq:ftoomch}. The $\sigma$-invariant measure $\mu$ satisfies $\mu=\frac{1}{\ell}\sum_{i=1}^\ell \mu_i$ and in particular is not ergodic with respect to $\sigma^\ell$.
\end{theorem}
It was shown in \cite{Mo17} that there exist examples of irreducible $N$-tuples $\mathsf{A}$ with a non-mixing equilibrium state for some $s>0$, such as the pair $\mathsf{A}=(A_1,A_2)\in M_2(\mathbb{R})^2$ defined by
\[A_1:=\left(\begin{array}{cc}0&2\\ 1&0\end{array} \right),\qquad A_2:=\left(\begin{array}{cc}0&1\\ 2&0\end{array} \right).\]
This pair lies in the class described in Theorem \ref{th:one} with $k=1$, $\ell=2$ and $X=\mathrm{Id}$. Theorem \ref{th:one} indicates that the mechanism by which mixing fails in this example is in fact the only possible mechanism for the failure of mixing in the invertible case.

In the presence of an invertible matrix Theorem \ref{th:one} allows us to positively answer two questions posed by the author in \cite{Mo17}:
\begin{corollary}\label{co:quod-erat-pedicandus}
Let $\mathsf{A}=(A_1,\ldots,A_N)\in M_d(\mathbb{R})^N$ be irreducible, let $s>0$, and suppose that at least one of the matrices $A_i$ is invertible. Let $\mu$ be the unique equilibrium state of $(\mathsf{A},s)$. Then the following are equivalent:
\begin{enumerate}[(i)]
\item
The measure $\mu$ is mixing with respect to $\sigma$.
\item
The measure $\mu$ is ergodic with respect to $\sigma^d$.
\item
For every $t>0$ the unique equilibrium state of $(\mathsf{A},t)$ is mixing with respect to $\sigma$.
\end{enumerate}
\end{corollary}
The hypothesis that one of the matrices is invertible plays only a small role in the proof of Theorem \ref{th:one} and plays no role at all in the deduction of Corollary \ref{co:quod-erat-pedicandus} from Theorem \ref{th:one}. Nonetheless, the following proposition demonstrates that  both Theorem \ref{th:one} and Corollary \ref{co:quod-erat-pedicandus} fail if this invertibility hypothesis is removed.
\begin{proposition}\label{pr:dont-swear-in-front-of-the-copy-editor}
Define a pair of matrices $\mathsf{A}=(A_1,A_2)\in M_3(\mathbb{R})^2$ by
\[A_1:=\left(\begin{array}{ccc}0&1&2\\2&0&0\\1&0&0\end{array}\right),\qquad A_2=\left(\begin{array}{ccc}0&2&1\\1&0&0\\2&0&0\end{array}\right).\]
Then $\mathsf{A}$ is irreducible, and for every $s>0$ the unique equilibrium state of $(\mathsf{A},s)$ is ergodic with respect to $\sigma^3$ but not with respect to $\sigma^2$. For every $s>0$ the matrices $\mathsf{A}_2$ have exactly two ergodic equilibrium states, the average of which is the equilibrium state of $(\mathsf{A},s)$.
\end{proposition}
At the time of writing we are not able to offer any further conjectures on the structure or behaviour of non-invertible irreducible matrix tuples which admit a non-mixing equilibrium state. 

\section{Proof of Theorem \ref{th:one}}

Let $\mathsf{A}=(A_1,\ldots,A_N)$ be as in Theorem \ref{th:one} and let $s>0$.  Let $\mu$ be the unique $\sigma$-invariant equilibrium state of $(\mathsf{A},s)$, which is ergodic with respect to $\sigma$. If $\mu$ can be expressed as a proper linear combination $\frac{1}{\ell}\sum_{i=1}^\ell\mu_i$ of distinct $\sigma^\ell$-invariant measures $\mu_i$ as in the statement of Theorem \ref{th:one} then it is not totally ergodic and is therefore not mixing. The difficult direction of the proof is thus to show that if $\mu$ is not mixing then the various properties described in the statement of the theorem hold true.
 
Suppose then that $\mu$ is not mixing. As noted in the introduction, by Theorem 5(ii) of \cite{Mo17} it is not totally ergodic. Let $n$ be the smallest integer such that $\mu$ is not ergodic with respect to $\sigma^n$ and note that obviously $n>1$. We note that the equilibrium state $\mu$ of $(\mathsf{A},s)$ is an equilibrium state of $(\mathsf{A}_n,s)$ which is invariant with respect to $\sigma^n$ but is not ergodic with respect to $\sigma^n$. It follows from Theorem \ref{th:fk} that $\mathsf{A}_n$ cannot be irreducible. Let $\mu_1,\ldots,\mu_m$ be a complete list of the  ergodic $\sigma^n$-invariant equilibrium states of $(\mathsf{A}_n,s)$, where we note by Theorem \ref{th:fk} that $m$ is finite.

We claim that in fact $m\geq n$. Let $p \geq 1$ be the smallest integer such that one of the measures $\mu_1,\ldots,\mu_m$ is $\sigma^p$-invariant, and pick $j$ such that $\sigma_*^p\mu_j=\mu_j$. Note that $\nu:=\frac{1}{p}\sum_{i=0}^{p-1}\sigma_*^i\mu_j$ is a $\sigma$-invariant equilibrium state of $(\mathsf{A},s)$ and therefore equals $\mu$. We cannot have $p=1$ since then $\mu$, which is not ergodic with respect to $\sigma^n$, would be equal to $\mu_j$, which is ergodic with respect to $\sigma^n$. It follows in particular that $\sigma_*\mu_j \neq \mu_j$, so $\nu$ is a proper linear combination of at least two distinct $\sigma^p$-invariant measures and as such cannot be ergodic with respect to $\sigma^p$. Since $\nu=\mu$ we see that $\mu$ is not ergodic with respect to $\sigma^p$ and therefore $p \geq n$ by the definition of $n$. Each of the $p$ measures $\sigma_*^i\mu_j$ is trivially $\sigma^n$-invariant and each must be an equilibrium state of $(\mathsf{A}_n,s)$ because otherwise their average $\mu$ could not be an equilibrium state of $(\mathsf{A}_n,s)$. Furthermore each is ergodic with respect to $\sigma^n$, since if we could write $\sigma_*^i\mu_j=\alpha \nu_1+(1-\alpha)\nu_2$ for two distinct $\sigma^n$-invariant measures $\nu_1$ and $\nu_2$ and some real number $\alpha \in (0,1)$ then $\mu_j$ would have a nontrivial decomposition $\mu_j=\alpha\sigma_*^{n-i}\nu_1 + (1-\alpha)\sigma_*^{n-i}\nu_2$ into distinct $\sigma^n$-invariant measures, contradicting the ergodicity of $\mu_j$ with respect to $\sigma^n$. Thus each $\sigma_*^i\mu_j$ is an ergodic $\sigma^n$-invariant equilibrium state of $(\mathsf{A}_n,s)$ and therefore must be equal to one of the measures $\mu_i$. It follows that if any two of the measures $\sigma_*^i\mu_j$ for $i=0,\ldots,p-1$ are not distinct then the minimality of $p$ is contradicted. Therefore there are at least $p$ distinct ergodic $\sigma^n$-invariant equilibrium states, which is to say $m \geq p$. We have shown that $m \geq p \geq n$.

Let  $k$ be the dimension of the lowest-dimensional nonzero linear subspace of $\mathbb{R}^d$ which is preserved by $\mathsf{A}_n$. Since $\mathsf{A}_n$ is reducible, $k<d$. Let $U\subseteq \mathbb{R}^d$ be any $k$-dimensional subspace which is preserved by $\mathsf{A}_n$ and define $U_1:=U$. Define a sequence of subspaces $U_i$ of $\mathbb{R}^d$ inductively by
\[U_{i+1}:=\mathrm{span} \,\mathsf{A}U_i:=\mathrm{span} \{A_ju \colon u \in U_i \text{ and }1 \leq j \leq N\}.\]
We claim that $\dim U_i=k$ for every integer $i\geq 1$. By hypothesis there exists $t \in \{1,\ldots,N\}$ such that $A_t$ is invertible. The case $i=1$ is obvious, so given $i > 1$ choose $r \geq 1$ such that $rn> i$, then $\mathsf{A}_{rn-i+1}\mathsf{A}_{i-1}U_1=\mathsf{A}_{rn}U_1\subseteq U_1$ since $\mathsf{A}_nU_1\subseteq U_1$ by definition of $U_1$. In particular $A_t^{rn-i+1}\mathsf{A}_{i-1}U_1\subseteq U_1$ and therefore $\mathsf{A}_{i-1}U_1\subseteq A^{-(rn-i+1)}_tU_1$. Since $A^{-(rn-i+1)}_tU_1$ is a $k$-dimensional subspace of $\mathbb{R}^d$ it follows that $\mathsf{A}_{i-1}U_1$ is contained in a $k$-dimensional subspace of $\mathbb{R}^d$ and therefore $U_i=\mathrm{span}(\mathsf{A}_{i-1}U_1)$ has dimension at most $k$. On the other hand we equally have $A_t^{i-1}U_1\subseteq \mathsf{A}_{i-1}U_1\subseteq U_i$ and therefore $U_i$ has dimension at least $k$. This proves the claim. 

Define now
\[V_i := \mathrm{span} \bigcup_{j=1}^i U_j\]
for each $i \geq 1$. Clearly  $\mathsf{A}U_i \subseteq U_{i+1}$ and therefore $\mathsf{A}V_i\subseteq V_{i+1}$  for each $i$. Clearly also $V_i\subseteq V_{i+1}$ for each $i$. It is impossible for $V_i$ to be a proper subspace of $V_{i+1}$ for every $i$ since then the dimension of $V_i$ would grow without limit as $i \to \infty$, which is impossible since $V_i \subseteq \mathbb{R}^d$. Therefore there must exist an integer $i$ such that $V_{i+1}=V_i$ and in particular $\mathsf{A}V_i \subseteq V_i$. Let $\ell$ be the smallest integer such that $\mathsf{A}V_\ell\subseteq V_\ell$. Since $\mathsf{A}$ is irreducible and $V_\ell$ contains the subspace $U_1$ which has nonzero dimension, this implies that $V_\ell=\mathbb{R}^d$.

We claim next that $U_{i+1}\cap V_i= \{0\}$ for every $i=1,\ldots,\ell-1$. Fix any  $i \in \{1,\ldots,\ell-1\}$. We observe that by construction $\mathsf{A}_nU_j \subseteq U_j$ for each $j$ and therefore in particular $U_{i+1}\cap V_i$ is an $\mathsf{A}_n$-invariant subspace. If $0<\dim (U_{i+1}\cap V_i)<k$ then the minimality of $k$ is contradicted. Clearly $\dim (U_{i+1}\cap V_i)\leq  \dim U_{i+1}=k$, so if $\dim (U_{i+1}\cap V_i)\geq k$ then necessarily $U_{i+1}\cap V_i = U_{i+1}$ and therefore $U_{i+1}\subseteq V_i$. It follows that $V_{i+1}=V_i$ and therefore $\mathsf{A}V_i \subseteq V_i$. In this case we must have $V_i=\mathbb{R}^d$ and $i<\ell$, but this contradicts the minimality of $\ell$. We conclude that necessarily $\dim (U_{i+1}\cap V_i)= 0$ for each $i=1,\ldots,\ell-1$, and this proves the claim. It follows by induction that $V_i=\bigoplus_{j=1}^i U_j$ for $i=1,\ldots,\ell$ and therefore $\mathbb{R}^d=V_\ell=\bigoplus_{i=1}^\ell U_i$. In particular $k\ell=d$. 

For $i=1,\ldots,\ell$ the $k$-dimensional subspace $U_i$ is clearly preserved by $\mathsf{A}_n$, and the minimality of $k$ implies that $\mathsf{A}_n$ does not preserve a proper subspace of any $U_i$. Since $\mathbb{R}^d=\bigoplus_{i=1}^\ell U_i$ it follows that $\mathsf{A}_n$ can be simultaneously block diagonalised with $\ell$ irreducible diagonal blocks each of size $k \times k$. By Theorem \ref{th:fk} it follows from this that there are at most $\ell$ distinct ergodic $\sigma^n$-invariant equilibrium states for $(\mathsf{A}_n,s)$, which is to say $\ell \geq m$. On the other hand $\mathsf{A}_nU_1\subseteq U_1$ but the $k$-dimensional subspaces spanned by each of $U_1, \mathsf{A}U_1,\ldots,\mathsf{A}_{\ell-1}U_1$ have trivial pairwise intersection and are therefore all distinct, so necessarily $n \geq \ell$. Combining the inequalities so far obtained yields $m\geq p \geq n \geq \ell \geq m$ and therefore $m,p,n$ and $\ell$ are all equal.

Let us now take stock of what has been proved. Since $n=m=\ell$ there are exactly $\ell$ distinct ergodic $\sigma^\ell$-invariant measures $\mu_1,\ldots,\mu_\ell$ which are equilibrium states for $(\mathsf{A}_\ell,s)$. The measure $\mu$ can be written as $\frac{1}{p}\sum_{i=0}^{p-1}\sigma_*^i\mu_j$ for a certain $\mu_j$, and the measures $\mu_j,\sigma_*\mu_j,\ldots,\sigma_*^{p-1}\mu_j$ are distinct ergodic  $\sigma^\ell$-invariant equilibrium states for $(\mathsf{A}_\ell,s)$; but since $p=\ell=m$, the $\ell$ distinct measures $\sigma_*^i\mu_j$ must be precisely the $\ell$ distinct measures $\mu_1,\ldots,\mu_\ell$, and therefore in fact $\mu=\frac{1}{\ell}\sum_{i=1}^\ell \mu_i$. We have $\mathbb{R}^d=\bigoplus_{i=1}^\ell U_i$, $\mathrm{span}(\mathsf{A}U_i) = U_{i+1}$ for $i=1,\ldots,\ell-1$, and $\mathsf{A}_nU_1\subseteq U_1$, but since $\ell=n$ it follows that $\mathrm{span}(\mathsf{A} U_\ell)=\mathrm{span}(\mathsf{A}_\ell U_1)=U_1$ so $\mathsf{A}U_\ell \subseteq U_1$. Thus the action of $\mathsf{A}$ cyclically permutes the subspaces $U_1,\ldots,U_\ell$ and therefore $\mathsf{A}$ admits the form \eqref{eq:form} after the application of a suitable change-of-basis matrix $X$. By direct calculation the matrices in $\mathsf{A}_n$ therefore admit the representation \eqref{eq:ftoomch}, and as noted previously each $k\times k$ block is irreducible. By Theorem \ref{th:fk} it follows that there is at most one equilibrium state for $(\mathsf{A}_\ell,s)$ associated to each diagonal block in the expression \eqref{eq:ftoomch}, but we have seen that there are exactly $\ell$ such equilibrium states and so each block contributes exactly one equilibrium state as claimed. The proof of the theorem is complete.

\section{Auxiliary lemmas}
Before proceeding with the proofs of Proposition \ref{pr:dont-swear-in-front-of-the-copy-editor} and Corollary \ref{co:quod-erat-pedicandus} we require two lemmas:
\begin{lemma}\label{le:other}
Let $\mathsf{A}=(A_1,\ldots,A_N)\in M_d(\mathbb{R})^N$ be irreducible and let $s>0$. Then
\[\limsup_{n \to \infty} \frac{1}{kn}\log\left(\sum_{i_1,\ldots,i_{kn}=1}^N \|A_{i_{kn}}\cdots A_{i_1}v\|^s\right) =P(\mathsf{A},s)\]
for every nonzero $v \in \mathbb{R}^d$ and integer $k \geq 1$.
\end{lemma}
\begin{proof}
Fix $s>0$ throughout the proof. For notational convenience define $a_n(v):=\sum_{i_1,\ldots,i_{n}=1}^N \|A_{i_n}\cdots A_{i_1}v\|^s$ for every $n \geq 1$ and $v \in \mathbb{R}^d$.
Let us first prove the lemma in the case $k=1$. Clearly the limit superior must less than or equal to $P(\mathsf{A},s)$, so to prove the lemma in this case we must show that the set
\[V:=\left\{v \in \mathbb{R}^d \colon \limsup_{n\to\infty} \frac{1}{n}\log a_n(v)<P(\mathsf{A},s)\right\}\]
contains only the zero vector. Using the elementary inequality
\[(|x|+|y|)^s \leq (2\max\{|x|,|y|\})^s\leq 2^s(|x|^s+|y|^s)\]
we have $a_n(u+\lambda v)\leq 2^sa_n(u)+(2|\lambda|)^sa_n(v)$ for all $u,v\in\mathbb{R}^d$, $\lambda \in \mathbb{R}$ and $n\geq 1$ and it follows that $V$ is a vector space. Obviously also $a_n(A_iv)\leq a_{n+1}(v)$ for every $i=1,\ldots,N$,  $n\geq 1$ and $v \in \mathbb{R}^d$ which implies that $V$ is $\mathsf{A}$-invariant. By irreducibility $V$ must therefore equal either $\{0\}$ or $\mathbb{R}^d$. Let us suppose for a contradiction that $V=\mathbb{R}^d$. If $e_1,\ldots,e_d \in \mathbb{R}^d$ denotes the standard basis, then
\begin{align*}P(\mathsf{A},s):=\lim_{n \to \infty} \frac{1}{n}\log \sup_{\|v\|=1} a_n(v) &= \limsup_{n \to \infty} \frac{1}{n}\log \left(\max_{\substack{\alpha_1,\ldots,\alpha_d \in\mathbb{R}\\ \sum_{i=1}^d|\alpha_i|^2=1}} a_n\left(\sum_{i=1}^d\alpha_i e_i\right)\right) \\
&\leq \limsup_{n\to\infty}\frac{1}{n}\log \left(d^s \max_{1 \leq i \leq d} a_n(e_i)\right)<P(\mathsf{A},s)\end{align*}
a contradiction. We conclude that necessarily $V=\{0\}$ and thus the lemma is proved in the case $k=1$.

Let us now show that $\limsup_{n \to \infty} (nk)^{-1}\log a_{nk}(v)$ is independent of $k\geq 1$. Fix nonzero $v \in \mathbb{R}^d$ and let $K:=\sum_{i=1}^N\|A_i\|^s>0$. If $n,k \geq 1$ and  $0\leq \ell<k$ then clearly 
\[a_{(n+1)k}(v)\leq K^{k-\ell}a_{nk+\ell}(v)\leq K^{k}a_{nk}(v),\]
so for fixed $\ell$ and $k$
\begin{align*}
\limsup_{n \to \infty} \frac{1}{nk}\log a_{nk}(v) &=\limsup_{n \to \infty} \frac{1}{nk}\log a_{(n+1)k}(v)\\
& \leq \limsup_{n \to \infty} \frac{1}{nk}\log a_{nk+\ell}(v) \leq \limsup_{n \to \infty} \frac{1}{nk}\log a_{nk}(v)\end{align*}
and therefore
\[\limsup_{n \to \infty} \frac{1}{nk}\log a_{nk}(v) =\limsup_{n \to \infty} \frac{1}{nk}\log a_{nk+\ell}(v)=\limsup_{n \to \infty} \frac{1}{nk+\ell}\log a_{nk+\ell}(v).\]
On the other hand for each fixed $k \geq 1$ we clearly have
\[\limsup_{n \to \infty} \frac{1}{n}\log a_{n}(v)=\max_{0 \leq \ell <k} \limsup_{n \to \infty} \frac{1}{nk+\ell}a_{nk+\ell}(v) = \limsup_{n \to \infty}\frac{1}{nk}\log a_{nk}(v)\]
and therefore the limit superior is independent of $k$ for each fixed $v$ as claimed.
\end{proof}
{\bf{Remark:}} the limit superior in Lemma \ref{le:other} is in fact a limit. Moreover one may show that there exists $K>0$ depending only on $\mathsf{A}$ and $s$ such that
\[K^{-1}e^{nP(\mathsf{A},s)}\|v\|^s\leq \sum_{i_1,\ldots,i_{n}=1}^N \|A_{i_{n}}\cdots A_{i_1}v\|^s \leq Ke^{nP(\mathsf{A},s)}\|v\|^s\]
for all $n \geq 1$ and $v \in \mathbb{R}^d$. This however requires substantially more effort to prove than the preceding lemma, which is sufficient for our purposes.

The following corollary of Lemma \ref{le:other} will be used in both the proof of Proposition \ref{pr:dont-swear-in-front-of-the-copy-editor} and the proof of Corollary \ref{co:quod-erat-pedicandus}:
\begin{lemma}\label{le:blox}
Let $\mathsf{A}=(A_1,\ldots,A_N)\in M_d(\mathbb{R})^N$ be irreducible and let $s>0$, and suppose that there exists an integer $n \geq 1$ such that $\mathsf{A}_n$ is block diagonal in the following sense: there exist natural numbers $d_1,\ldots,d_\ell$ such that $\sum_{i=1}^\ell d_i=d$, and matrices $A_{(i_1,\ldots,i_n)}^{(j)}$ such that 
\[A_{i_n}\cdots A_{i_1} = \left(\begin{array}{ccccc}
A_{(i_1,\ldots,i_n)}^{(1)} &0 & 0& \cdots &0 \\
0 & A_{(i_1,\ldots,i_n)}^{(2)}  & 0 & \cdots &0\\
0 & 0 & A_{(i_1,\ldots,i_n)}^{(3)} & \cdots &0\\
\vdots & \vdots & \vdots  & \ddots &\vdots \\
0 & 0 & 0  & \cdots &A_{(i_1,\ldots,i_n)}^{(\ell)} 
\end{array}\right)\]
for every $i_1,\ldots,i_n \in \{1,\ldots,N\}$, where each matrix $A_{(i_1,\ldots,i_n)}^{(j)}$ is of dimension $d_j \times d_j$. Define $\mathsf{A}_n^{(j)}$ to be the $N^n$-tuple of all possible matrices $A_{(i_1,\ldots,i_n)}^{(j)}$, and suppose that each $\mathsf{A}_n^{(j)}$ is irreducible. Then $P\left(\mathsf{A}_n^{(j)},s\right)=P(\mathsf{A}_n,s)$ for every $j=1,\ldots,\ell$.
\end{lemma}
\begin{proof}
Let us write $\mathbb{R}^d=\bigoplus_{j=1}^\ell U_j$ where for each $j=1,\ldots,\ell$ the space $U_j$ is preserved by all of the matrices $A_{i_n}\cdots A_{i_1}$ and has dimension $d_j$, and the restriction of $A_{i_n}\cdots A_{i_1}$ to $U_j$ is given by the matrix $A_{(i_1,\ldots,i_n)}^{(j)}$. Let $v \in U_j$ be nonzero, then
\[\limsup_{m \to \infty} \frac{1}{m}\log \left(\sum_{i_1,\ldots,i_{nm}=1}^N \left\|A_{(i_{(m-1)n+1},\ldots,i_{mn})}^{(j)} \cdots A_{(i_1,\ldots,i_n)}^{(j)}v\right\|^s\right)= P\left(\mathsf{A}^{(j)}_n,s\right)\]
by applying Lemma \ref{le:other} to the irreducible $N^n$-tuple $\mathsf{A}_n^{(j)}$, and equally
\[\limsup_{m \to \infty} \frac{1}{m}\log \left(\sum_{i_1,\ldots,i_{nm}=1}^N \left\|A_{(i_{(m-1)n+1},\ldots,i_{mn})}^{(j)} \cdots A_{(i_1,\ldots,i_n)}^{(j)}v\right\|^s\right)\]
\[=\limsup_{m \to \infty} \frac{1}{m}\log \left(\sum_{i_1,\ldots,i_{nm}=1}^N \left\|A_{i_{mn}}\cdots A_{i_1}v\right\|^s\right)= nP(\mathsf{A},s)=P(\mathsf{A}_n,s)\]
by applying Lemma \ref{le:other} to the irreducible $N$-tuple $\mathsf{A}$.
\end{proof}

\section{Proof of Proposition \ref{pr:dont-swear-in-front-of-the-copy-editor}}

We may now prove Proposition \ref{pr:dont-swear-in-front-of-the-copy-editor}.
Fix $s>0$ throughout the proof. Let us show that the pair $\mathsf{A}:=(A_1,A_2)$ defined by 
\[A_1:=\left(\begin{array}{ccc}0&1&2\\2&0&0\\1&0&0\end{array}\right),\qquad A_2:=\left(\begin{array}{ccc}0&2&1\\1&0&0\\2&0&0\end{array}\right)\]
is irreducible. If $\{A_1,A_2\}$ preserves a two-dimensional subspace of $\mathbb{R}^2$ then the pair $\{A_1^T,A_2^T\}$ preserves its one-dimensional orthogonal complement, but since $A_1=A_2^T$ this implies that $\{A_1,A_2\}$ also preserves that same one-dimensional space. It follows that if $\mathsf{A}$ preserves a proper nonzero subspace of $\mathbb{R}^3$ then it necessarily preserves a one-dimensional subspace of $\mathbb{R}^3$. 

Suppose then that $\mathsf{A}$ is reducible and so there exists a one-dimensional subspace which is preserved by $\mathsf{A}$. By one-dimensionality the restrictions of $A_1$ and $A_2$ to that subspace must commute, so the invariant one-dimensional subspace is annihilated by $A_1A_2-A_2A_1$. But we have
\begin{equation}\label{eq:bits}A_1A_2=\left(\begin{array}{ccc}5&0&0\\0&4&2\\0&2&1\end{array}\right),\qquad A_2A_1=\left(\begin{array}{ccc}5&0&0\\0&1&2\\0&2&4\end{array}\right)\end{equation}
and therefore
\[A_1A_2-A_2A_1= \left(\begin{array}{ccc}0&0&0\\0&3&0\\0&0&-3\end{array}\right).\]
Thus the only one-dimensional subspace of $\mathbb{R}^3$ which is annihilated by $A_1A_2-A_2A_1$ is the first co-ordinate axis, but this space is transparently not an invariant subspace of $\mathsf{A}$. We conclude that no one-dimensional invariant subspace exists and therefore $\mathsf{A}$ is irreducible as claimed. In particular there exists a unique equilibrium state $\mu$ for $(\mathsf{A},s)$ and $\mu$ is ergodic with respect to $\sigma$. 

We next wish to show that $\mu$ is ergodic with respect to $\sigma^3$. To see this it is sufficient to show that $\mathsf{A}_3$ is irreducible, since then $(\mathsf{A}_3,s)$ has a unique equilibrium state $\nu$, say; since $\nu$ must be ergodic with respect to $\sigma^3$, and $\mu$ is an equilibrium state of $(\mathsf{A}_3,s)$, we have $\mu=\nu$ by uniqueness so that $\mu$ is ergodic with respect to $\sigma^3$. To show that $\mathsf{A}_3$ is irreducible it is sufficient to show that $\{A_1^3,A_2^3\}$ is irreducible: but in fact $A_i^3$ is just a nonzero scalar multiple of $A_i$ for each $i=1,2$ by direct calculation (or by applying the Cayley-Hamilton Theorem together with the observation that $\tr A_i=\det A_i=0$ and $A_i$ is not nilpotent since $A_i^2$ has nonzero trace). The irreducibility of $\mathsf{A}_3$ and the ergodicity of $\mu$ with respect to $\sigma^3$ therefore follow.

We now notice that since the matrices $A_1A_2$, $A_2A_1$ have the form \eqref{eq:bits}, and since
\[A_1^2=\left(\begin{array}{ccc}4&0&0\\0&2&1\\0&4&2\end{array}\right),\qquad A_2^2=\left(\begin{array}{ccc}4&0&0\\0&2&4\\0&1&2\end{array}\right),\]
the $4$-tuple $\mathsf{A}_2=(A_1^2,A_1A_2,A_2A_1,A_2^2)$ is reducible. Let us define
\[\mathsf{B}:=(B_{11},B_{12},B_{21},B_{22})=(4,5,5,4) \in M_1(\mathbb{R})^4\]
and
\[\mathsf{C}:=(C_{11},C_{12},C_{21},C_{22}) \in M_2(\mathbb{R})^4\]
where
\[C_{11}=\left(\begin{array}{cc}2&1\\4&2\end{array}\right),\quad C_{12}=\left(\begin{array}{cc}4&2\\2&1\end{array}\right),\quad C_{21}=\left(\begin{array}{cc}1&2\\2&4\end{array}\right),\quad C_{22}=\left(\begin{array}{cc}2&4\\1&2\end{array}\right)\]
so that $A_iA_j=B_{ij}\oplus C_{ij}$ for each $i,j \in \{1,2\}$.  Since $C_{11}$ has eigenbasis $(1,2)^T$, $(1,-2)^T$ and $C_{22}$ has eigenbasis $(2,1)^T$, $(-2,1)^T$ the matrices $C_{11}$ and $C_{22}$ have no common invariant subspace and therefore $\mathsf{C}$ is irreducible.  By Lemma \ref{le:blox} we therefore have $P(\mathsf{B},s)=P(\mathsf{C},s)=P(\mathsf{A}_2,s)$. 

Let $s>0$ and note that by Theorem \ref{th:fk} $\mathsf{B}$, $\mathsf{C}$ have one equilibrium state each, which we denote by $\mu_{\mathsf{B}}$ and $\mu_{\mathsf{C}}$ respectively. As before we view $\mu_{\mathsf{B}}$ and $\mu_{\mathsf{C}}$ as $\sigma^2$-invariant measures on $\Sigma_2$. We claim that $\mu_{\mathsf{B}} \neq \mu_{\mathsf{C}}$. Suppose for a contradiction that $\mu_{\mathsf{B}}=\mu_{\mathsf{C}}$. By Theorem \ref{th:fk} there exist constants $K_1,K_2>0$ such that 
\[ K_1^{-1}\mu_{\mathsf{B}}([i_1\cdots i_{2n}]) \leq e^{-nP(\mathsf{B},s)}\|B_{i_{2n}i_{2n-1}}B_{i_{2n-2}i_{2n-3}}\cdots B_{i_2i_1}\|^s \leq K_1\mu_{\mathsf{B}}([i_1\cdots i_{2n}]) \]
\[ K_2^{-1}\mu_{\mathsf{C}}([i_1\cdots i_{2n}]) \leq e^{-nP(\mathsf{C},s)}\|C_{i_{2n}i_{2n-1}}C_{i_{2n-2}i_{2n-3}}\cdots C_{i_2i_1}\|^s \leq K_2\mu_{\mathsf{C}}([i_1\cdots i_{2n}]) \]
for every $n\geq 1$ and $i_1,\ldots,i_{2n} \in \{1,2\}$. Applying this with $n=2k$ and with $(i_1,\ldots,i_{2n})=(1,2,2,1,1,2,2,1,\ldots,1,2,2,1)$ yields
\[K_2K_1^{-1} \leq \frac{\|(B_{12}B_{21})^k\|^s}{\|(C_{12}C_{21})^k\|^s} \leq K_1K_2^{-1} \]
since $P(\mathsf{B},s)=P(\mathsf{C},s)$ and we have assumed that $\mu_{\mathsf{B}}=\mu_{\mathsf{C}}$. This implies
\[\rho(C_{12}C_{21})=\lim_{k \to \infty} \|(C_{12}C_{21})^k\|^{\frac{1}{k}}=\lim_{k \to \infty} \|(B_{12}B_{21})^k\|^{\frac{1}{k}} =\rho(B_{12}B_{21})\]
by Gelfand's formula, where $\rho(A)$ denotes the spectral radius of $A$; but straightforward calculation shows that actually $\rho(B_{12}B_{21})=25>16=\rho(C_{12}C_{21})$, a contradiction. We conclude that $\mu_{\mathsf{B}} \neq \mu_{\mathsf{C}}$ as claimed.

It remains to show that $\mu$ is not ergodic with respect to $\sigma^2$. Since $\mu_{\mathsf{B}}$ and $\mu_{\mathsf{C}}$ are distinct from one another, at least one of them is not equal to $\mu$. Call this measure $\nu$. Since $\nu \neq \mu$, $\nu$ cannot be $\sigma$-invariant, since it would then be an equilibrium state of $(\mathsf{A},s)$ and equal $\mu$ by uniqueness. Therefore $\nu':=\frac{1}{2}\left(\nu + \sigma_*\nu\right)$ is not equal to  $\nu$. Clearly $\nu'$ is $\sigma$-invariant and is an equilibrium state of $(\mathsf{A},s)$, so $\nu'=\mu$ by uniqueness. Since $\mu=\frac{1}{2}\nu + \frac{1}{2}\sigma_*\nu$ is a $(\mathsf{A}_2,s)$-equilibrium state both $\nu$ and $\sigma_*\nu$ must be equilibrium states of $(\mathsf{A}_2,s)$. Moreover the ergodicity of $\nu$ with respect to $\sigma^2$ implies the ergodicity of $\sigma_*\nu$ with respect to $\sigma^2$ in a manner similar to the proof of Theorem \ref{th:one}. Since $\nu \neq \mu=\nu'$ the measures $\nu$ and $\sigma_*\nu$ must be distinct. We have shown that $\nu$ and $\sigma_*\nu$ are distinct $\sigma^2$-invariant equilibrium states of $(\mathsf{A}_2,s)$ which are ergodic with respect to $\sigma^2$, so they must equal $\mu_{\mathsf{B}}$ and $\mu_{\mathsf{C}}$ in some order. In particular $\mu=\frac{1}{2}\left(\mu_{\mathsf{B}}+\mu_{\mathsf{C}}\right)$ as indicated in the statement of the proposition. The proof of Proposition \ref{pr:dont-swear-in-front-of-the-copy-editor} is complete.

\section{Proof of Corollary \ref{co:quod-erat-pedicandus}}

 If $\mu$ is not ergodic with respect to $\sigma^d$ then it is not totally ergodic and in particular is not mixing. If on the other hand $\mu$ is not mixing, then by Theorem \ref{th:one} it is a proper linear combination of $\sigma^\ell$-invariant measures for some factor $\ell$ of $d$. In particular it is a proper linear combination of $\sigma^d$-invariant measures and is therefore not ergodic with respect to $\sigma^d$. This establishes (i)$\iff$(ii). The implication (iii)$\implies$(i) is trivial. In order to prove (i)$\implies$(iii) we first require a result from \cite{Mo17}:
 \begin{lemma}\label{le:anne}
 Let $\mathsf{A}=(A_1,\ldots,A_N) \in M_{d_1}(\mathbb{R})^N$ and $\mathsf{B}=(B_1,\ldots,B_N)\in M_{d_2}(\mathbb{R})^N$ be irreducible and suppose that $P(\mathsf{A},s)=P(\mathsf{B},s)$ for every $s>0$. For each  $s>0$ let $\mu_s$ and $\nu_s$ denote the unique equilibrium state of $(\mathsf{A},s)$ and $(\mathsf{B},s)$ respectively. Then $\mu_s=\nu_s$ for some $s>0$ if and only if $\mu_s=\nu_s$ for all $s>0$.
 \end{lemma}
 \begin{proof}
 This follows from the equivalence of statements (i) and (iii) of Theorem 9 in \cite{Mo17}.
 \end{proof}
 
 Let us now prove (i)$\implies$(iii). If $\mu$ is not mixing then by Theorem \ref{th:one} there exist $\ell>1$ distinct ergodic equilibrium states for $(\mathsf{A}_\ell,s)$, one of which comes from each of the $\ell$ irreducible diagonal blocks in \eqref{eq:ftoomch}. By Lemma \ref{le:blox} the pressure $P(\cdot,t)$ of each of these $\ell$ diagonal blocks is equal to $P(\mathsf{A}_\ell,t)$ for every parameter value $t>0$. It follows by Theorem \ref{th:fk} that for each $t>0$, each of the $\ell$ diagonal blocks of $\mathsf{A}_\ell$ contributes  one ergodic equilibrium state of $(\mathsf{A}_\ell,t)$. By Lemma \ref{le:anne} these $\ell$ measures must be distinct for every $t>0$ since they are distinct for $t=s$. If $t>0$ is fixed, let $\nu_1,\ldots,\nu_\ell$ be the distinct ergodic equilibrium states of $(\mathsf{A}_\ell,t)$. By the irreducibility of $\mathsf{A}$ there is a unique equilibrium state $\nu$ for $(\mathsf{A},t)$. If any $\nu_j$ is $\sigma$-invariant then it must be an equilibrium state of $(\mathsf{A},t)$ and hence must equal $\nu$. Since $\ell>1$ and the measures $\nu_i$ are distinct there therefore exists $j$ such that $\nu_j \neq \nu$ and in particular such that $\sigma_*\nu_j \neq \nu_j$. The measure $\frac{1}{\ell}\sum_{i=0}^{\ell-1}\sigma_*^i\nu_j$ is a $\sigma$-invariant equilibrium state for $\mathsf{A}_\ell$ and is therefore equal to the equilibrium state $\nu$ of $(\mathsf{A},t)$. Since $\sigma_*\nu_j \neq \nu_j$, it follows that $\nu=\frac{1}{\ell}\sum_{i=0}^{\ell-1}\sigma_*^i\nu_j$ is a proper linear combination of at least two distinct $\sigma^\ell$-invariant measures, and hence is not ergodic with respect to $\sigma^\ell$. In particular $\nu$ is not totally ergodic with respect to $\sigma$ and is therefore not mixing with respect to $\sigma$. Since $t>0$ was arbitrary this completes the proof of (iii) and hence completes the proof of the corollary.
 
\bibliographystyle{acm}
\bibliography{mixoe}

\end{document}